\documentclass[12pt]{article}

\usepackage[margin=1in]{geometry}
\usepackage{graphicx}

\usepackage{amsthm,amsmath,amssymb}
\usepackage{cite}
\usepackage[colorlinks=true,citecolor=black,linkcolor=black,urlcolor=blue]{hyperref}

\usepackage{enumerate}

\theoremstyle{plain}
\newtheorem{thm}{Theorem}
\newtheorem{lemma}[thm]{Lemma}
\newtheorem{corollary}[thm]{Corollary}

\newtheorem*{theoremaux}{Theorem \theoremauxnum}
\newtheorem*{unlabeledthm}{Theorem}
\gdef\theoremauxnum{1}

\newtheorem{conj}{Conjecture}

\def\Q{\mathbb{Q}}
\def\PP{\mathbb{P}}

\def\lex{{\operatorname{lex}}}
\def\cstd{\operatorname{cstd}}

\def\Sym{\boldsymbol{\operatorname{Sym}}}
\def\QSym{\boldsymbol{\operatorname{QSym}}}
\def\Sm {\mathfrak{S}}
\def\Des{\operatorname{Des}}

\def\leqlex{ \leq_\lex  }

\def\lesslex{ <_\lex  }

\title{\bf  Chromatic symmetric functions of hypertrees  }
\author{Jair Taylor\thanks{The author was partially supported by the National Science Foundation grant DMS-1101017.}}

\date{\today}

\begin{document}
\maketitle
\begin{abstract}
The chromatic symmetric function $X_H$ of a hypergraph $H$ is the generating function for all colorings of $H$ so that no edge is monochromatic.  When $H$ is an ordinary graph, it is known that $X_H$ is positive in the fundamental quasisymmetric functions $F_S$, but this is not the case for general hypergraphs.  We exhibit a class of hypergraphs $H$ --- hypertrees with prime-sized edges --- for which $X_H$ is $F$-positive, and give an explicit combinatorial interpretation for the $F$-coefficients of $X_H$.
\end{abstract}
\begin{section}{Introduction}\label{sec:intro}

In \cite{chromaticsymmetric}, Stanley defined the \emph{chromatic symmetric function} of a graph $G$, and since then this invariant has been an object of much study \cite{gasharov, wachschromatic1, wachschromatic2, prietocaterpillars, saganchromatic}.  A \emph{proper coloring} of an ordinary graph $G = (V,E)$ is a map $\chi: V \rightarrow \PP = \{1,2, \ldots\}$ so that if $\{u,v\} \in E$ then $\chi(u) \neq \chi(v)$.  When $V$ is finite, we define the chromatic symmetric function $X_G$ to be a formal power series in the commuting indeterminates $x_1, x_2, \ldots$ given by taking the sum over all proper colorings of $G$ where a vertex colored $i$ is assigned the weight $x_i$.  That is, $$X_G = \sum_{\chi   }  \prod_{v \in V} x_{\chi(v)}$$ with the sum over all proper colorings $\chi$ of $G$.

A formal power series (over, say, $\Q$) of bounded degree in $x_1, x_2, \ldots$  is called a \emph{symmetric function} if it remains the same after any permutation of its variables.  The chromatic symmetric function $X_G$ is indeed symmetric, since the properness of a coloring is preserved under any permutation of the set of colors $\PP$.  It is natural, then, is to consider the expansion of $X_G$ in various bases of the ring of symmetric functions $\Sym$, and there are a number of conjectures and open problem concerning positivity of $X_G$ in these bases.  

Our present interest is in the larger ring $\QSym$ of \emph{quasisymmetric functions}.  A formal power series of bounded degree $X$ in the variables $x_1, x_2, \ldots$ is called \emph{quasisymmetric} if the coefficient of $x_{i_1}^{\alpha_1} x_{i_2}^{\alpha_2} \cdots x_{i_k}^{\alpha_k}$ in $X$ is the same as the coefficient of $x_{j_1}^{\alpha_1} x_{j_2}^{\alpha_2} \cdots x_{j_k}^{\alpha_k}$ in $X$ whenever $i_1 < \ldots < i_k$ and $j_1 < \ldots <j_k$.  The vector space $\QSym_n$ of quasisymmetric functions of degree $n$ has dimension $2^{n-1}$ and has the basis of \emph{fundamental quasisymmetric functions} $F^n_{S}$ indexed by subsets $S \subseteq [n-1]$, defined by
$$F^n_S = \sum_{i_1, i_2, \ldots, i_n} x_{i_1} x_{i_2} \cdots x_{i_n}$$
with the sum over all weakly increasing sequences $i_1 \leq i_2\leq \ldots\leq i_n$ of positive integers with the restriction that if $j \in S$ then $i_j < i_{j+1} $.  In what follows all the symmetric and quasisymmetric functions are homogeneous, so we will write $F_S = F_S^n$ without ambiguity.
If $X$ is a symmetric function then it is also quasisymmetric, so we may consider the coefficients $a_{S}$ in the expansion $X = \sum_{S\subseteq [n-1]} a_{S} F_{S}$.  If each $a_{S}$ is nonnegative, we will say that $X$ is $F$-positive.  In \cite{chromaticsymmetric}, Stanley used the theory of $P$-partitions to show that $X_G$ is always $F$-positive; in that case the $F$-coefficients count linear extensions of posets defined by acyclic orientations of $G$.

In \cite{hypersymmetric}, Stanley presented a generalization of the chromatic symmetric function to \emph{hypergraphs}.   A hypergraph is a graph where the edges are allowed to contain more than two elements; that is, a hypergraph is a pair $H = (V,E)$, $V$ finite, where $E$ is a family of subsets of $V$ called the hyperedges (or just edges) of $H$ with $|e| > 1$ for each $e \in E$.  We say $e \in E$ is \emph{monochromatic} under the coloring $\chi: V \rightarrow \PP$ if $\chi$ maps all of $e$ to a single color $i$, so that $\chi(v) = i$ for all $v \in e$, and a coloring $\chi$ of $H$ is \emph{proper} if no edge of $H$ is monochromatic under $\chi$.  It might seem more natural to require that each $v \in e$ have a different color. However, as Stanley notes in \cite{hypersymmetric}, these colorings would be the same as the proper colorings of the ordinary graph obtained by replacing each hyperedge $e \in E$ by the complete graph on $e$, so nothing new would be gained by considering hypergraphs.

The chromatic symmetric function $X_H$ is then defined, as before, by 
$$X_H = \sum_{\chi}  \prod_{v \in V} x_{\chi(v)}$$
where the sum is taken over all proper colorings $\chi$ of $H$.

Again $X_H$ is symmetric, but unlike in the case of ordinary graphs, $X_H$ is not always $F$-positive.  For example, if $H = (V,E)$ where $V = \{1,2,3,4\}$ and $E = \{\{1,2,3\}, \{2,3,4\}\}$ then 
\begin{align}\label{exampleXH}
X_H = 2F_{ \{ 1\}} + 6F_{\{2\}} + 2F_{\{3\}} + 4F_{\{1,2\}} + 8F_{\{1,3\}} + 4F_{\{2,3\}} - 2F_{\{1,2,3\}}
\end{align}
is not $F$-positive.  The reader might observe that the coefficients in \eqref{exampleXH} sum to $24 = 4!$, and this is not a coincidence.  The sum of the $F$-coefficients in a chromatic symmetric function $X_H$ will always be $n!$ where $n = |V|$, and this can be seen by considering the coefficient of $x_1x_2\ldots x_n$.  Thus when $X_H$ is $F$-positive, we might expect then to be able to write $X_H$ as a sum of fundamental quasisymmetric functions indexed by permutations, and this is what we proceed to do for a certain class of hypergraphs: the \emph{hypertrees} with prime-sized edges.  There are a number of closely-related definitions of \emph{hypertree} occurring in the literature; we adopt the definition given in \cite{gesselkalikowhypertrees}.  


Let $H = (V,E)$ be a hypergraph.  A \emph{path} in $H$ is a nonempty sequence $$v_1, e_1, v_2, e_2, \ldots, e_m, v_{m+1}$$ where each $e_i \in E$ with $v_i, v_{i+1} \in e_i$ and the edges $e_i$ and vertices $v_i$ of the path are distinct, except that we allow $v_1 = v_{m+1}$.  If $v_1 = v_{m+1}$ and $m > 1$ we say the path is a \emph{cycle}.  We say that a hypergraph $H = (V,E)$ is \emph{connected} if there is a path from $v$ to $v'$ for any given $v, v' \in V$.  A \emph{hypertree} is a hypergraph that is connected and has no cycles.  Thus in a hypertree there is a \emph{unique} path between any two distinct vertices.  A hypergraph $H$ is called \emph{linear} if $|e \cap e'| \leq 1$ for any distinct edges $e, e' \in E$.  Hypertrees are linear, for if there are distinct $v_1, v_2 \in e_1 \cap e_2$ for $e_1 \neq e_2$ then there is a cycle $v_1, e_1, v_2, e_2, v_1$.  Figure 1 depicts a hypertree.

\begin{figure}[htbp]\label{fig:hypertreepic}
	\vspace{-30pt}
  \begin{centering}
    \includegraphics[width=.6\textwidth ]{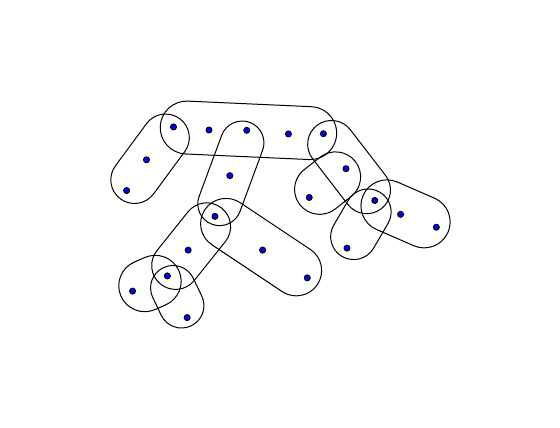}
	\vspace{-40pt}
    \caption{A hypertree with edges $e \in E$ circled.}
  \end{centering}
\vspace{-10pt}
\end{figure}

Our main result is the following fact, appearing as Theorem \ref{bigthm} in Section \ref{sec:comb}.
\begin{unlabeledthm}
Let $H = (V,E)$ be a hypertree so that $|e|$ is a prime number for each edge $e \in E$.  Then $X_H$ is $F$-positive.  In particular,
$$X_H = \sum_{\pi \in \Sm_V}  F_{\Des_H(\pi)}$$
where $n = |V|$ and $\Des_H(\pi)$ is the set of \emph{$H$-descents} of the permutation $\pi$, to be defined in Section \ref{sec:comb}.
\end{unlabeledthm}
It is not true that $X_H$ is $F$-positive whenever $H$ is linear.  For example, $X_H$ is not $F$-positive when $H$ consists of the edges $\{1,2,3\},\{1,4\},\{2,4\},\{3,4,5\}$.  On the other hand, we guess (Conjecture \ref{conj:hypertreeconjecture}) that the primality condition is not necessary for $F$-positivity, although our proof relies on primality in a crucial way.  We also note that it is easy to extend Theorem \ref{bigthm} to disjoint unions of hypertrees, or \emph{hyperforests}, but for simplicity we only consider the connected case. 

In Section \ref{sec:singleedge}, we use a standardization procedure due to Gessel and Reutenauer \cite{gesselnecklaces} to show $F$-positivity of $X_H$ when $H$ consists of a single prime-sized edge.  In Section \ref{sec:Fpos}, we combine the result of Section \ref{sec:singleedge} with the theory of \emph{$P$-partitions} due to Stanley \cite{stanleyppartitions} and Gessel\cite{gesselquasi} to show $F$-positivity of $X_H$ for a hypertree $H$ with prime-size edges.  In Section \ref{sec:comb} we describe a combinatorial interpretation of the results in Section \ref{sec:Fpos} by giving the definition of $H$-descents and proving Theorem \ref{bigthm}.  We conclude in Section \ref{sec:end} by giving some conjectures and suggestions for further work.
\end{section}
\begin{section}{The single edge case}\label{sec:singleedge}
If $H= (V,E)$ only has a single edge $e$ consisting of all of $V$, so that $E = \{V\}$, then a proper coloring of $H$ is any coloring that is not constant.  Thus if $|V| = n$ then $X_H = p_1^n - p_n$, where $p_i$ is the $i$th \emph{power sum symmetric function} $$p_i = x_1^i +x_2^i + \cdots.$$  In this case it is not difficult to show that $X_H$ is $F$-positive using standard results on symmetric functions. In fact, $X_H$ is \emph{Schur-positive}, which implies $F$-positivity.  This in itself is not enough to show $F$-positivity for other chromatic symmetric functions of hypergraphs.  However, in this section we will prove $F$-positivity of $X_H = p_1^n - p_n$ when $n$ is prime by exhibiting an explicit partition of the set of all nonconstant colorings where each set in the partition has a generating function that is a fundamental quasisymmetric function $F_S$.  We will see in Section \ref{sec:Fpos} that if such a partition can be found for each edge $e$ in a hypertree then there is a similar partition of the set of proper colorings of that hypertree.

If $V$ is a finite set with $|V| = n$, let $\Sm_V$ be the set of permutations of $V$ realized as bijections $\pi: V\rightarrow [n]$; if $V = [n]$ we write $\Sm_n$ for $\Sm_{[n]}$.  Given a permutation $\pi \in \Sm_V$ and a subset $S \subseteq [n-1]$, let $A(\pi, S)$ be the set of colorings $\chi: V \rightarrow \PP$ satisfying the conditions $$\chi(\pi^{-1}(1)) \leq  \chi(\pi^{-1}(2)) \leq \ldots \leq  \chi(\pi^{-1}(n))$$ and $\chi(\pi^{-1}(i)) < \chi(\pi^{-1}(i+1))$ when $i \in S$.  Thus for any $\pi \in \Sm_V$ and $S\subseteq [n-1]$, $A(\pi, S)$ has the quasisymmetric generating function $$\sum_{\chi \in A(\pi, S)} x_{\chi(\pi^{-1}(1))} x_{\chi(\pi^{-1}(2))} \cdots x_{\chi(\pi^{-1}(n))} = F_S.$$

Recall that the \emph{descent set} $\Des(\pi)$ of a permutation $\pi \in \Sm_n$ is the set of $i < n$ so that $\pi(i) > \pi(i+1)$.  The goal of this section is to prove the following fact.
\begin{thm}\label{partition}
Let $V$ be a set with $|V| = n$ prime and let $c: V \rightarrow V$ be a cyclic permutation of $V$.  Then the set of nonconstant colorings $\chi: V \rightarrow \PP$ is the disjoint union $$\biguplus_{\pi\in \Sm_V} A(\pi, \Des(\pi c \pi^{-1})).$$
\end{thm}
Theorem \ref{partition} immediately gives the $F$-expansion of the symmetric function $p_1^n - p_n$ when n is prime; we have 
\begin{align}\label{singleedgeFexpansion}
p_1^n - p_n = \sum_{\pi \in \Sm_n} F_{\Des(\pi c \pi^{-1})}.
\end{align}
If we think of each $\pi \in \Sm_V$ as a \emph{labeling} of $V$ with the labels $1, \ldots, n$, identifying $v$ with $\pi(v)$, then $\pi c \pi^{-1}$ is the same as $c$ when viewed as a permutation of the labels.  Each cyclic permutation of the labels $[n]$ will appear $n$ times in this sum, so we have
\begin{align}\label{cycles}
\frac{p_1^n - p_n}{n} = \sum_{c} F_{\Des(c)}
\end{align}
where the sum is taken over all cyclic permutations $c: [n] \rightarrow [n]$.  The identity \eqref{cycles} was shown by Gessel and Reutenauer \cite{gesselnecklaces}, and in fact if $n$ is prime then $(p_1^n - p_n)/n$ is both the generating function for \emph{primitive necklaces} of length $n$ and the Frobenius chracteristic of the $\Sm_n$-representation given by the degree-$n$ multilinear part of the free Lie algebra on a set of size $n$.  

In the proof that follows it will be convenient to assume without loss of generality that $V = [n]$ and $c$ is the particular cyclic permutation $c(i) = i+1$ for $1 \leq i <n$ with $c(n) = 1$. We think of a coloring $\chi: V \rightarrow \PP$ as a \emph{word} $w = \chi$, where we write $w = w(1) \cdots w(n) \in \PP^n$.  To prove Theorem \ref{partition}, we will use a method of obtaining a permutation $\pi \in \Sm_n$ from a nonconstant word $w \in \PP^n$ when $n$ is  prime due to Gessel and Reutenauer \cite{gesselnecklaces}.  Let $w = w(1)w(2) \cdots w(n) \in \PP^n$ be a word that uses at least two distinct letters from $\PP$, so that we do not have $w(1) = w(2) = \cdots = w(n)$.  Let $r_i(w)$ be the rotation $$r_i(w) =  w(i) w(i+1) \cdots w(n) w(1) w(2) \cdots w(i-1)$$ of $w$.  The rotations $r_1(w), r_2(w), \ldots, r_n(w)$ need not be distinct in general.  For example, if $w = 1212$ then $r_1(w) = r_3(2) = 1212$.  If $n$ is prime, however, this cannot occur and the rotations of $w$ are all distinct as long as $w$ is a nonconstant word.

Assuming $n$ is prime, define the \emph{cyclic standardization} of $w$, which we'll denote $\cstd(w)$, to be the permutation obtained by ordering these rotated words lexicographically: we say $\pi = \cstd(w)$ if $\pi$ is the unique permutation in $\mathfrak{S}_n$ so that $r_x(w) \lesslex r_y(w)$ whenever $\pi(x) < \pi(y)$.  That is, we find $\pi = \cstd(w)$ by setting $\pi(i) = j$ when $r_i(w)$ is the $j$th smallest rotation of $w$.  For example, if $w = 2114132$ then $\pi = \cstd(w) = 4137265$; we have $\pi(2) = 1$ since $r_2(w) = 1141321$ is the least rotation of $w$ lexicographically, $\pi(5) = 2$ since $r_5(w) = 1322114$ is the next smallest, etc.

By the primality of $n$, the set of words $w \in \PP^n$ that are not constant is the disjoint union
$$\biguplus_{\pi \in \mathfrak{S}_n} \{ w \in \PP^n: \cstd(w) = \pi \}.$$ 
\noindent Thus the proof of Theorem \ref{partition} is immediate from the following lemma.

\begin{lemma}
Let $n$ be prime and let $w \in \PP^n$ be a nonconstant word. Then $\cstd(w) = \pi$ if and only if $w \in A(\pi, S)$ where $S = \Des(\pi c \pi^{-1})$.
\end{lemma}
\begin{proof}

First, suppose that $\cstd(w) = \pi$.  For any $1 \leq i < n$, suppose $\pi(x) = i$ and $\pi(y) = i+1$; then $r_y(w) = w(y)w(y+1)\cdots$ is the next largest rotation of $w$ in lexicographic order after $r_x(w) = w(x)w(x+1)\cdots$, where we take $n+1 = 1$, $n+2 = 2$, etc.  In particular, we must have $w(x) \leq w(y)$, or $w(\pi^{-1}(i)) \leq w(\pi^{-1}(i+1) )$ as desired.  Now suppose that $i$ is a descent of $\pi c \pi^{-1}$; we must show that $w(x) <w(y)$.  We have $\pi c \pi^{-1}(i) > \pi c \pi^{-1}(i+1)$; that is, $\pi(x +1 ) > \pi(y+1)$.  Then we have 
\begin{align*}
r_{x+1}(w) = w(x+1) w(x+2) \cdots >_\lex r_{y+1}(w) = w(y+1) w(y+2) \cdots .
\end{align*}
But we also know that $w(x) w(x+1) \cdots \lesslex w(y)w(y+1) \cdots$, and the only way both of these lexicographic inequalities can occur is if $w(x) < w(y)$.  

Conversely, suppose that $w \in A(\pi, S)$.  To show that $\cstd(w) = \pi$, we will show that $\pi(x) < \pi(y)$ implies $r_x(w) \lesslex r_y(w)$.  Since $n$ is prime, all rotations of $w$ are distinct, and so it is enough to show that $\pi(x) < \pi(y)$ implies $r_x(w) \leqlex r_y(w)$; we may also assume that $\pi(x) = i$ and $\pi(y) = i+1$.  We will proceed inductively.  For any word $v = v(1)v(2) \cdots v(n) \in \PP^n$, let $v|_m$ be the truncation $v(1)\cdots v(m)$.  We will show that for each $m$, if $\pi(x) < \pi(y)$ we must have $r_x(w)|_m \leqlex r_y(w)|_m$.  If $m=1$, the truncations $r_x(w)|_m,r_y(w)|_m$ are the single-character words $w(x),w(y)$; and $w(x) = w(\pi^{-1}(i)) \leq w(\pi^{-1}(i+1)) = w(y)$ since $w \in A(\pi, S)$.

Now suppose that the statement holds for $m$.  We will show that $r_{x}(w)|_{m+1} \leqlex r_{y}(w)|_{m+1}$.    By the argument for the base case, we know $w(x) \leq w(y)$; if $w(x) < w(y)$ we are done, so assume $w(x) = w(y)$.  Since $w \in A(\pi, S)$,  $i$ must be an ascent of $\pi c \pi^{-1}$; that is, $\pi c \pi^{-1}(i) < \pi c \pi^{-1}(i+1)$, or $\pi( x+1) < \pi(y+1)$.  By our inductive hypothesis, we must have $r_{x+1}(w)|_m \leqlex r_{y+1}(w)|_m$.  Then $r_x(w)|_{m+1} = w(x)r_{x+1}(w)|_m \leqlex w(x)r_{y+1}(w)|_m = r_y(w)|_{m+1}$.
\end{proof}
\end{section}
\begin{section}{Proof of F-positivity}\label{sec:Fpos}
Now we are in a position to show the $F$-positivity of $X_H$ when $H = (V,E)$ is a hypertree with prime-sized edges.  The primality gives us the decomposition described in Theorem \ref{partition} for each edge $e \in E$, and the hypertree structure will enable us to glue these decompositions together to get a similar decomposition of the set of all proper colorings of $H$.  The glue, in this case, is the theory of \emph{$P$-partitions}.

Given a poset $P$ on a vertex set $V$, a mapping $f: V \rightarrow \PP = \{1,2,\ldots\}$ is a $P$-partition if $x \leq_P y$ implies $f(x) \leq f(y)$.  If $P$ is the poset $[n]$ with the usual order, then a $P$-partition is a sequence of increasing integers $f(1) \leq f(2) \leq f(3) \leq \ldots \leq f(n)$.  This is equivalent to the usual definition of a \emph{partition} of the integer $f(1) + f(2) + \cdots +f(n)$.  Traditionally a partition of an integer is written in descending order, so what we call $P$-partitions were called \emph{reverse} $P$-partitions by Stanley \cite{stanleyppartitions, EC2}. 

Suppose that $\omega \in \Sm_V$ is a bijection $V \rightarrow [n]$.  In what follows it will be convenient to identify $\omega$ with the total order $<_\omega$ put on the vertices of $P$ where $x<_\omega y$ means that $\omega(x) < \omega(y)$.  A \emph{ $(P, \omega)$-partition} is a $P$-partition $f$ that has strict inequalities where the orders $P$ and $\omega$ disagree.  That is, if $x <_P y$ and $x <_\omega y$ then $f(x) \leq f(y)$, but if $x <_P y$ and $x >_\omega y$ then $f(x)<f(y)$.  

A \emph{linear extension} of $P$ is an order-preserving bijection $f: V \rightarrow [n]$.  The main result on $(P,\omega)$-partitions we need is the following fact, sometimes called the Fundamental Theorem of $(P,\omega)$-Partitions. See \cite[Lemma 3.15.3]{EC1} for a proof when $(P,\omega)$-partitions are taken to be order-reversing; it is given without proof in \cite[7.19.4]{EC2} for $(P,\omega)$-partitions taken to be order-preserving as we do.
\begin{thm}\label{fund}
Let $P  = (V, \leq_P)$ be a finite poset with $|V| = n$ and let $\omega: V \rightarrow [n]$ be any bijection.  Then the set of $(P, \omega)$-partitions is exactly the disjoint union $$\biguplus_\pi A(\pi, \Des(\omega \pi^{-1}))$$
where $A(\pi, S)$ is as defined in Section \ref{sec:singleedge}, and the union is taken over all linear extensions $\pi: V \rightarrow [n]$ of $P$.
\end{thm}

The key fact we will use about hypertrees is that posets on different edges are compatible with each other.

\begin{lemma}\label{lemma:compatible}
Let $H = (V,E)$ be a hypertree with $E = \{e_1, \ldots, e_k\}$.  Suppose that each edge $e_i \in E$ has an associated poset $P_i$ with vertex set $e_i$ and relation $<_i$.  Define the relation $<$ on $V$ by taking the transitive closure of all the relations $<_e$, so that $x < y$ in $V$ if there is a chain $x = v_1 <_{i_1} v_2 <_{i_2} \cdots <_{i_l} v_{l} = y$.  Then $P = (V, <)$ is a poset.
\end{lemma}
\begin{proof}
Form a directed graph $G$ on $V$ by setting $x \rightarrow y$ when there is an edge $e \in E$ with $x,y \in e$ and $x <_e y$.  Then $G$ is easily seen to be acyclic since $H$ is a hypertree, and any directed acyclic graph determines a poset after extending transitively.
\end{proof}

\begin{thm}\label{thm:Fpos}
Let $H = (V,E)$ be a hypertree so that $|e|$ is prime for each $e \in E$.  Then $X_H$ is $F$-positive.
\end{thm}
\begin{proof}
Say $E = \{e_1, \ldots, e_k\}$.  For each edge $e_i \in E$ fix a particular bijection $c_i: e_i \rightarrow e_i$ that is cyclic.  Since each edge $e_i \in |E|$ has $|e_i|$ prime, by Theorem \ref{partition} the set of nonconstant colorings $\chi: e_i \rightarrow \PP$ is the disjoint union
\begin{align}\label{smallunion}
\biguplus_{\pi \in \Sm_{e_i}} A(\pi, \Des(\pi c_i \pi^{-1})).
\end{align}
Let $PC(H)$ be the set of proper colorings $\chi$ of $H$.  A coloring $\chi$ of $H$ is proper if and only if each restriction $\chi|_{e_i}: e_i \rightarrow \PP$ is not constant, so we have
\begin{align}\label{bigunion}
PC(H) = \biguplus_{\pi_1, \ldots, \pi_k} A(\pi_1, \pi_2, \ldots, \pi_k)
\end{align}
where the union is taken over all $k$-tuples $(\pi_1, \ldots, \pi_k)$ with $\pi_i \in \Sm_{e_i}$ and 
$$A(\pi_1, \pi_2, \ldots, \pi_k) = \{ \chi:\PP\rightarrow V: \chi|_{e_i} \in A(\pi_{i}, \Des(\pi_i c_{i} \pi_i^{-1})) \text{ for all } i\}.$$
The union \eqref{bigunion} is disjoint since each union \eqref{smallunion} is: a coloring $\chi \in A(\pi_1, \pi_2, \ldots, \pi_k)$ uniquely determines each $\pi_i$ since the restriction $\chi|_{e_i}$ uniquely determines $\pi_{i}$.

Given an edge $e_i$ and a bijection $\pi_i: e_i \rightarrow [m] \in \Sm_{e_i}$ where $m = |e_i|$, define a poset $P_{\pi_i}$ (actually a total order) on the vertex set $e_i$ by $x <_{P_{\pi_i}} y$ when $\pi_i(x) <\pi_i(y)$, and let $\omega_{\pi_i}: e_i \rightarrow [m]$ be the labeling $\pi_i c_i$.  Then $\pi_i$ is the unique linear extension of $P_{\pi_i}$, so by Theorem \ref{fund} the set of $(P_{\pi_i},\omega_{\pi_i})$-partitions is exactly the set $A(\pi_i, \Des( \omega_{\pi_i} \pi_i^{-1} )) = A(\pi_i, \Des( \pi_i c_i \pi_i^{-1} ))$.  Thus $A(\pi_1, \pi_2, \ldots, \pi_k)$ is the set of colorings $\chi$ so that $\chi|_{e_i}$ is a $(P_{\pi_i},\omega_{\pi_i})$-partition for each $i$.

Fix a choice of $\pi_i \in \Sm_{e_i}$ for each edge $e_i$.  By Lemma \ref{lemma:compatible} there are well-defined posets $P_{\pi_1, \ldots, \pi_k}, Q_{\pi_1, \ldots, \pi_k}$ given by taking the transitive closure of the relations of the posets $P_{\pi_i}, \omega_{\pi_i}$ respectively.  Let $\omega_{\pi_1, \ldots, \pi_k}$ be any linear extension of $Q_{\pi_1, \ldots, \pi_k}$.  We claim that $A(\pi_1, \pi_2, \ldots, \pi_k)$ is exactly the set of $(P_{\pi_1, \ldots, \pi_k}, \omega_{\pi_1, \ldots, \pi_k})$-partitions.  It is clear from the definitions that if $\chi$ is a $(P_{\pi_1, \ldots, \pi_k}, \omega_{\pi_1, \ldots, \pi_k})$-partition then $\chi|_{e_i}$ is a $(P_{\pi_i}, \omega_{\pi_i})$-partition for each $i$, so $\chi \in  A(\pi_1, \pi_2, \ldots, \pi_k)$.  Conversely, if $\chi \in  A(\pi_1, \pi_2, \ldots, \pi_k)$ it is not hard to see that $\chi$ is a $(P_{\pi_1, \ldots, \pi_k}, \omega_{\pi_1, \ldots, \pi_k})$-partition by repeatedly applying the ``local conditions'' that each $\chi|_{e_i}$ is a $(P_{\pi_i},\omega_{\pi_i})$-partition.  

Combining Theorem \ref{fund} with \eqref{bigunion} then gives 
\begin{align}\label{eq:Fpos}
PC(H) = \biguplus_{\pi_1, \ldots, \pi_k} \biguplus_\sigma A(\sigma, \Des \omega_{\pi_1, \ldots, \pi_k} \sigma^{-1} )
\end{align}
where the union is taken over all tuples $\pi_1, \ldots, \pi_k$ with $\pi_i \in e_i$ and linear extensions $\sigma: V \rightarrow [n]$ of $P_{\pi_1, \ldots, \pi_k}$.  Since the quasisymmetric generating function of $A(\sigma, S)$ is $F_S$ we are done.
\end{proof}

By rewriting \eqref{eq:Fpos} in a simpler form we can give an expression for $X_H$ as a sum of fundamental quasisymmetric functions indexed by permutations $\pi: V \rightarrow [n]$.

\begin{corollary}[of the proof of Theorem \ref{thm:Fpos}]  \label{cor:nicer} 
Let $H = (V,E)$ be a hypertree.  Given a permutation $\pi: V \rightarrow [n]$, define posets $P(\pi), Q(\pi)$ on $V$ so that when $x,y$ both belong to the same edge $e$ then $x<_{P(\pi)} y$ iff $\pi(x) <\pi(y)$ and $x <_{Q(\pi)} y$ iff $\pi c (x)  < \pi c (y)$. Fix a linear extension $\omega_\pi$ of $Q(\pi)$ for each $\pi \in \Sm_V$.  Then 
\begin{align}\label{eq:nicer}
X_H = \sum_{\pi \in \Sm_V} F_{\Des( \omega_\pi \pi^{-1})}.
\end{align}
\end{corollary}
\begin{proof}
Say $E = \{e_1, \ldots, e_k\}$.  Then for each $\pi: V \rightarrow [n] \in \Sm_V$ there is a unique choice of $\pi_1, \ldots, \pi_k$ so that $\pi$ is a linear extension of $P_{\pi_1, \ldots, \pi_k}$: let $\pi_i = g \circ \pi|_{e_i}$ where $g$ is the unique increasing function from $\pi(e_i)$ to $[|e_i|]$.  Applying this fact to \eqref{eq:Fpos} and taking the quasisymmetric generating function gives \eqref{eq:nicer}.
\end{proof}

In Corollary \ref{cor:nicer} the choice of the linear extension $\omega_\pi$ is arbitrary.  Every poset has a linear extension, and this fact is enough to prove $F$-positivity.  From a combinatorial standpoint, however, it would be desirable to find a specific choice of $\omega_\pi$ that is natural in some sense. In the next section we do so, giving a simple combinatorial interpretation to the $F$-coefficients of $X_H$.

\end{section}
\begin{section}{Combinatorial interpretation}\label{sec:comb}
Before we can define the \emph{$H$-descents} alluded to in the introduction, we need to show the existence of a particularly nice ordering of the edges of a hypertree.  The following lemma says that any hypertree may be constructed by adding one edge at a time, each new edge intersecting the others in a single vertex.
\begin{lemma}\label{lemma:niceorder}
Let $H = (V,E)$ be a hypertree.  Then there is an ordering of its edges so that $E = \{e_1,e_2, \ldots, e_k\}$ with
\begin{align}\label{eq:niceorder}
|(e_1 \cup e_2 \cup \cdots \cup e_i) \cap e_{i+1}| = 1 \text{ for } i=1, \ldots, k-1.
\end{align}
\end{lemma}
\begin{proof}
It is enough to find an $e \in E$ and $v \in e$ so that $e' \cap e \subseteq \{v\}$ for any $e' \in E$ with $e' \neq e$.  Once such an $e$ is found, let $H' = (V', E')$ with $V' = V \backslash e \cup \{v\}$, $E' = E \backslash \{e\}$.  It is easy to check that $H'$ is a hypertree with $k-1$ edges, so we may assume inductively that $H'$ has an ordering $e_1, e_2, \ldots, e_{k-1}$ of $E'$ satisfying \eqref{eq:niceorder}. Setting $e_k = e$, we see that $e_1, \ldots, e_k$ is the desired order of $E$.  

To find such an $e$ and $v \in e$, let $v_1,f_1, v_2, f_2, \ldots, v_l,f_l,v_{l+1}$ be a path of maximal length $l$ in $H$.  We claim that $e = f_l$, $v = v_l$ satisfies the desired property.  Suppose that there is $e' \in E$ with $e' \cap f_l \not\subseteq \{v_l\}$; then there is $u \in e' \cap f_l$ with $u \neq v$.  We must have $e' = f_j$ for some $j < l$, or else we would have a longer path $v_1, f_2, \ldots, v_l, f_l, u, e', u'$ where $u \in e'$ with $u \neq u'$.  Say that $j$ is as large as possible.  Then we have a cycle $u, f_j, v_{j+1},f_{j+1}, \ldots, f_l,  u$ which violates the definition of a hypertree.  
\end{proof}
In fact, the converse of Lemma \ref{lemma:niceorder} is easily seen to hold as well, so that the existence of edge-orderings satisfying $\eqref{eq:niceorder}$ characterizes hypertrees.  From now on we will generally assume that $H$ is equipped with some choice of such an edge-ordering and our subsequent definitions are all based on this edge-ordering.

Suppose that $H = (V,E)$ is a hypertree with $E = \{e_1, \ldots, e_k\}$ satisfying \eqref{eq:niceorder}, and fix a choice of cyclic permutation $c_i: e_i \rightarrow e_i$ for each edge $e \in E$.  Suppose also that $V = [n]$, so that $V$ is equipped with the order $<$.  Since $H$ is a hypertree, for each $i$ there is a \emph{unique} path $i = v_{1}, e_{j_1}, v_2, e_{j_2}, \ldots, e_{j_l}, v_{l+1} = i+1$ from $i$ to $i+1$, where the vertices and edges in the path are all distinct.  Let $j_r = \min (j_1, j_2, \ldots, j_l)$.  Then we say that $i$ is an $H$-descent if $c_{j_r}(v_r) > c_{j_r}(v_{r+1})$. 

\begin{figure}[htbp]
	\vspace{-30pt}
  \begin{centering}
    \includegraphics[width=.6\textwidth ]{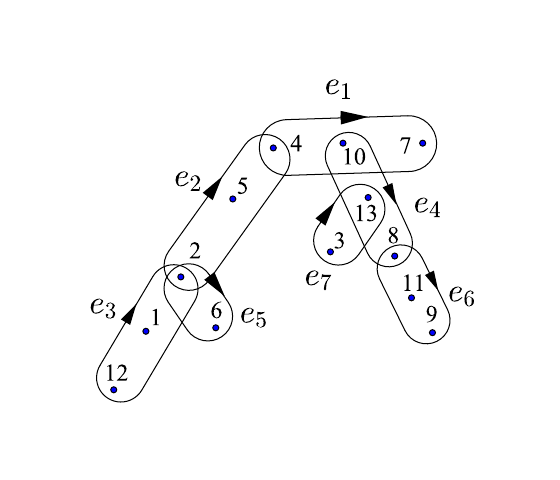}
	\vspace{-30pt}
    \caption{A hypertree with labeled vertices, a suitable ordering of its edges, and a cyclic permutation of each edge.}
  \end{centering}
\vspace{-10pt}
\label{fig:labeledhypertree}
\end{figure}

For example, let $H$ be the hypergraph in Fig. 2.  The cyclic permutations $c_i$ are given by reading along the indicated direction in cycle notation, so that $c_{3}$ is $(12,1,2)$ with $12 \mapsto 1 \mapsto 2 \mapsto 12$.  The unique path from $1$ to $2$ is just $1, e_3, 2$ since $1$ and $2$ are both contained in $e_3$.  Then $c_3(1) = 2 < c_3(2) = 12$, so $1$ is not an $H$-descent.  The unique path from $2$ to $3$ is given by $2, e_2, 4, e_1, 10, e_4, 13, e_7, 3$ and the edge with the smallest index occurring in this path is $e_1$.  Then $c_1(4) = 10 > c_1(10) = 7$ and so $2$ is an $H$-descent.  Continuing, we find the $H$-descents are $\{2, 6, 8, 10, 12\}$.

Now let $H = (V,E)$ be a hypertree as before with an edge-ordering satisfying \eqref{eq:niceorder} and a cyclic permutation of each edge, but now allow $V$ to be an arbitrary finite set which we will think of as being unordered.  Given a permutation $\pi: V \rightarrow [n] \in \Sm_V$, we will consider $\pi$ a \emph{labeling} of $V$, identifying $v$ with $\pi(v)$.  We then denote the corresponding set of $H$-descents by $\Des_H(\pi)$ and call them the $H$-descents of $\pi$.  With these definitions in hand, we state our main theorem.
\begin{thm}\label{bigthm}
Let $H = (V,E)$ be a hypertree so that $|e|$ is prime for each edge $e \in E$.  Fix an ordering of the edges so $E = \{e_1, \ldots, e_k\}$ with the property that $|(e_{1} \cup \cdots \cup e_{i})\cap e_{i+1}| = 1$ for all $1 \leq i < k$, and also fix a choice of cyclic permutation $c_i: e_i\rightarrow e_i$ of each edge $e_i \in E$. Then $$X_H = \sum_{\pi\in \Sm_V} F_{\Des_H(\pi)}$$ where $\Des_H(\pi)$ is the set of $H$-descents of $\pi$ with respect to the chosen edge-ordering and cyclic permutations.
\end{thm}
\noindent Note that in the case where $H$ consists of a single edge $e$ with a cyclic permutation $c: e \rightarrow e$, $\Des_H(\pi)$ is exactly $\Des (\pi^{-1} c \pi)$, so Theorem \ref{bigthm} reduces to Corollary \ref{singleedgeFexpansion}.

To prove Theorem \ref{bigthm}, we will need a systematic way of combining total orders together.  Given totally ordered sets $(U, \omega_U)$, $(V, \omega_V)$ where $U,V$ share a single element, say $U \cap V = \{x\}$, we define $\omega_U \leftarrow \omega_V$ to be the total order of the union $U \cup V$ given by ``inserting'' $V$ with its total order $\omega_V$ into the place of $x$ in $U$.  That is, $\omega$ is the unique total order agreeing with $\omega_U,\omega_V$ on $U,V$ so that when $u \in U, v \in V$ we have $u <_\omega v$ if and only if $u <_{\omega_U} x$.  Thus if $U$ consists of elements $u_1 <_{\omega_U} u_2 <_{\omega_U}\cdots <_{\omega_U} u_m$, with $u_i = x$, and $V$ has elements $v_1 <_{\omega_V} v_2 <_{\omega_V}\cdots <_{\omega_V} v_n$, then $\omega = \omega_U \leftarrow \omega_V$ is the total order of $U\cup V$ with $$u_1 <_\omega u_2  <_\omega \cdots  <_\omega u_{i-1}  <_\omega v_1  <_{\omega}\cdots <_{\omega} v_n  <_{\omega} u_{i+1}  <_{\omega} \cdots  <_{\omega} u_m.$$ For example, if $U = \{x <_{\omega_U} b <_{\omega_U} y\}$, $V = \{a <_{\omega_V} b <_{\omega_V} c\}$ are totally ordered sets then $\omega = \omega_U \leftarrow \omega_V$ is the total order $x <_\omega a <_{\omega} b <_{\omega} c <_\omega y$.

We now consider the total orders that arise from repeated insertion.

\begin{lemma}\label{lemma:path}
Let $H = (V,E)$ be a hypertree with $E = \{e_1, \ldots, e_k\}$ so that \eqref{eq:niceorder} holds.  Suppose there is a total order $\omega_i$ on each $e_i$, and define a total order $\omega$ on $V$ by 
$$\omega =  (\cdots(\omega_1 \leftarrow \omega_2) \leftarrow \cdots ) \leftarrow \omega_k.$$
Then for any distinct $x,y \in V$, $x <_\omega y$ if and only $v_r <_{\omega_{j_r}} v_{r+1}$ where 
\begin{align}\label{eq:path}
x = v_1, e_{j_1}, v_2, e_{j_2}, \ldots, v_l, e_{j_l}, v_{l+1} = y
\end{align}
is the unique path from $x$ to $y$ in $H$ and $j_r = \min (j_1, j_2, \ldots, j_l)$.
\end{lemma}
\begin{proof}
We proceed by induction on the number of edges of $H$.  If $H$ has only one edge, the statement is trivial, so suppose that the statement holds for hypertrees with fewer than $k$ edges and that $H$ has exactly $k$ edges.  Let $H' = (V', E')$ be the hypertree with $V' = e_1 \cup \cdots \cup e_{k-1}$ and $E' = E\backslash \{e_k\}$, and let $\omega'$ be the total order on $V'$ given by $$\omega' =  (\cdots(\omega_1 \leftarrow \omega_2) \leftarrow \cdots ) \leftarrow \omega_{k-1},$$ so that $\omega = \omega' \leftarrow \omega_k$.  If both $x$ and $y$ are in $V'$ then we are done by the inductive hypothesis.  Similarly if $x,y \in e_k$ then there is nothing to show.  So assume that $x \in e_k \backslash V'$ and $y \in V' \backslash e_k$ and let \eqref{eq:path} be the path from $x$ to $y$, so that $\{v_2\} = \left(e_1 \cup \cdots \cup e_{k-1} \right) \cap e_{k}$.   From the definition of the insertion $\omega' \leftarrow \omega_k$ we have $x <_{\omega} y$ if and only if $v_2 <_{\omega'} y$.  Then $$ v_2, e_{j_2}, \ldots, v_l, e_{j_l}, v_{l+1}$$ is the unique path from $v_2$ to $y$ in $H'$ and clearly $j_r = \min (j_1, j_2, \ldots, j_l) = \min (j_2, \ldots, j_l)$ since $j_1 = k$ is the highest index of any edge in $H$.  Thus by our inductive hypothesis we see that $v_2 <_{\omega'} y$ is equivalent to $v_r <_{\omega_{j_r}} v_{r+1}$.
\end{proof}

\begin{proof}[Proof of Theorem \ref{bigthm}]
Given a bijection $\pi: V \rightarrow [n]$, let $P(\pi)$ and $Q(\pi)$ be as in the statement of Corollary \ref{cor:nicer}.  For each $i$ let $\omega_i$ be the total order of $e_i$ given by restricting $Q$ to $e_i$, so that $x <_{\omega_i} y$ in $e$ if $\pi c_i(x) < \pi c_i(y)$, and let $\omega_\pi$ be the total order on $V$ given by $\omega_\pi =  (\cdots(\omega_1 \leftarrow \omega_2) \leftarrow \cdots ) \leftarrow \omega_k.$  Then $\omega_\pi$ is a linear extension of $Q(\pi)$, and by Lemma \ref{lemma:path} we see that $i$ is a descent of $\omega_\pi \pi^{-1}$ if and only if $v_r >_{\omega_{j_r}} v_{r+1}$, that is, $\pi c_{j_r} (v_r) > \pi c_{j_r}(v_{r+1})$ where $j_r$ is the least-index edge in the path $\pi^{-1}(i) = v_1, e_{j_1}, \ldots, e_{j_l}, v_{l+1} = \pi^{-1}(i+1)$ from $\pi^{-1}(i)$ to $\pi^{-1}(i+1)$.  After identifying $v$ with $\pi(v)$, the descents of $\omega_\pi \pi^{-1}$ become the $H$-descents, so that $\Des_H(\pi) = \Des(\omega \pi^{-1})$.  Applying Corollary \ref{cor:nicer} then finishes the proof.
\end{proof}

\end{section}
\begin{section}{Suggestions for future work}\label{sec:end}
It is likely that the condition that the edges have prime size could be removed.  A closer examination of the proof of $F$-positivity (Theorem \ref{thm:Fpos}) reveals that it does not depend on primality \emph{per se}, but only on the existence of partitions of colorings of the form
\begin{align}\label{coxeterpartition}
\{\chi: [e_i] \rightarrow \PP: \chi \text{ not constant}\} = \biguplus_{\pi \in \Sm_{e_i}} A(\pi, S(\pi)). 
\end{align}
for each edge $e_i$ of $H$, where $S(\pi) \subseteq [n-1]$ is a choice of subset for each $\pi \in \Sm_{e_i}$, with $n = |e_i|$.  We need only to give the role played by the maps $\pi_i c_i$ in the proof of Theorem \ref{thm:Fpos} to an appropriate choice of bijections $\omega_{\pi_i} \in \Sm_{e_i}$ so that $\Des( \omega_{\pi_i} \pi_i^{-1}) = S( \pi )$ for each $\pi_i \in S_{e_i}$.  Thus we have:
\begin{thm}
Suppose that $H = (V,E)$ is a hypertree so that for each $e \in E$ there is a partition of the form \eqref{coxeterpartition}.  Then $X_H$ is $F$-positive.
\end{thm}
The fact that a partition of the form \eqref{coxeterpartition} exists when $n= |e_i|$ is prime gives a proof of $F$-positivity of hypertrees with prime-sized edges.  In fact, such a partition of the nonconstant colorings of a set of $n=4$ elements does exist as well; it was found with a search algorithm using the software package Sage \cite{sage}.  Finding such a partition for each $n$ would then constitute a proof of the following.
\begin{conj}\label{conj:hypertreeconjecture}
Let $H$ be a hypertree.  Then $X_H$ is $F$-positive.  
\end{conj}
We can rephrase this idea in terms of \emph{simplicial complexes}. A simplicial complex $\Delta$ is a family of subsets of a finite vertex set $V$ so that if $F \in \Delta$ and $F' \subseteq F$ then $F' \in \Delta$.  If $\Delta$ is a simplicial complex and $S\subseteq \Delta$ is any subset of $\Delta$, then $S$ is a \emph{partial simplicial complex} and we say that $S$ is \emph{partitionable} if $S$ is a disjoint union $$S = \biguplus_i [G_i, F_i]$$ where the $F_i$ are facets (maximal faces) of $\Delta$, $G_i \subseteq F_i$, and $[G_i, F_i] = \{ F \in \Delta: G_i \subseteq F \subseteq F_i\}$.  Then the existence of a partition of the nonconstant colorings of the form \eqref{coxeterpartition} when $|e_i| = n$ is equivalent to the statement that $\Delta_n \backslash \{\emptyset\}$ is partitionable where $\Delta_n$ is the \emph{Coxeter complex} of type $A_{n-1}$, a simplicial complex whose facets are in natural bijection with permutations $\pi \in \Sm_n$.  The problem of partitionability for a partial simplicial complex $S \subseteq \Delta_n$ is discussed by Breuer and Klivans in \cite{scheduling}, where $S$ is thought of as a \emph{scheduling problem}.

The \emph{Schur functions} form an important basis of $\Sym$ with deep connections to representation theory of the symmetric and general linear groups.  The optimistic reader might hope that chromatic symmetric functions of hypertrees would in fact be Schur-positive, but this is not the case even for ordinary graphs.  For example, if $C = (V,E)$ is the ``claw'' with $V= \{1,2,3,4\}, E = \{\{1,2\}, \{1,3\}, \{1,4\}\}$ then $X_C$ is not Schur-positive.  Stanley has conjectured in \cite{chromaticsymmetric} that if $G$ is \emph{clawfree} then $X_G$ is Schur-positive, where a graph $G$ is clawfree when it has no induced subgraphs isomorphic to the claw $C$. 

It would be interesting to generalize Stanley's conjecture to hypergraphs, but we do not attempt that here.  Instead we offer a more modest conjecture.  We say that a hypergraph $H$ is an \emph{interval hypergraph} if it is isomorphic to a hypergraph $(V,E)$ where $V = [n]$ and each edge $e \in E$ is an interval $e = \{i,i+1, \ldots, j\}$.  Recall that a hypergraph $H$ is linear if $|e \cap e'| \leq 1$ for each pair of distinct edges $e, e'$. For example, if $V = [9]$ and $E = \{\{1,2,3\}, \{3,4,5\},\{5,6\}, \{6,7,8,9\} \}$ then $H = (V,E)$ is an interval hypergraph that is linear.
\begin{conj}\label{verylinear}
If $H$ is a linear interval hypergraph then $X_H$ is Schur-positive.
\end{conj}
\noindent Note that connected linear interval hypergraphs are hypertrees, so they are at least $F$-positive when they have prime-sized edges.  In the case when $H=G$ is a linear interval hypergraph that is in an ordinary graph, $G$ is just a disjoint union of paths; in that case $X_G$ was shown to be $e$-positive by Stanley \cite{chromaticsymmetric} and hence Schur-positive.

Conjecture \ref{verylinear} was motivated by the study of \emph{formal group laws}. A one-dimensional, commutative formal group law in characteristic $0$ is equivalent to a formal power series of the form 
\begin{align}\label{fgl}
f(f^{-1}(x_1) + f^{-1}(x_2) + \cdots)
\end{align}
when $f(x)$ is a formal power series in one variable $x$ with $f(0) = 0, f'(0) = 1$. In \cite{FGLs}, the author gives a number of examples of generating functions $f(x)$ for which the formal group law \eqref{fgl} can be written as a sum of chromatic symmetric functions of certain linear interval hypergraphs.

\end{section}
\section*{Acknowledgements}
The author is indebted to Ira Gessel for suggesting this line of research and for many helpful ideas along the way.  We would also like to thank Sara Billey, Felix Breuer, Caroline Klivans,  Brendan Pawlowski, Jos\'{e} Samper, Josh Swanson and Greg Warrington for their useful insights and feedback.

\end{document}